\documentclass[11pt]{amsart}
\usepackage[utf8]{inputenc}

\usepackage{amsthm}
\usepackage{amsfonts}         
\usepackage{amsmath}
\usepackage{amssymb}
\usepackage{enumerate}
\usepackage[normalem]{ulem}
\usepackage{comment}
\usepackage[margin=1in,footskip=0.25in]{geometry}
\usepackage{hyperref}

\usepackage{xcolor}

\newcommand{\R}{\mathbb{R}}
\newcommand{\C}{\mathbb{C}}

\newcommand{\N}{\mathbb{N}}

\newcommand{\Z}{\mathbb{Z}}

\newcommand{\conv}{\operatorname{Conv}}
\newcommand{\supp}{\operatorname{supp}}
\renewcommand{\d}[1]{\ensuremath{\operatorname{d}\!{#1}}}
\DeclareMathOperator{\sign}{sign}

\DeclareMathOperator{\tr}{tr}
\DeclareMathOperator{\Var}{Var}

\DeclareMathOperator{\length}{length}

\newcommand*\pFqskip{8mu}
\catcode`,\active
\newcommand*\pFq{\begingroup
        \catcode`\,\active
        \def ,{\mskip\pFqskip\relax}%
        \dopFq
}
\catcode`\,12
\def\dopFq#1#2#3#4#5{%
        {}_{#1}F_{#2}\biggl[\genfrac..{0pt}{}{#3}{#4};#5\biggr]%
        \endgroup
}

\theoremstyle{plain}
\newtheorem{thm}[subsection]{Theorem}
\newtheorem{lem}[subsection]{Lemma}
\newtheorem{prop}[subsection]{Proposition}
\newtheorem{conj}[subsection]{Conjecture}
\newtheorem{defn}[subsection]{Definition}
\newtheorem{kor}[subsection]{Corollary}

\theoremstyle{definition}

\newtheorem{kys}[subsection]{Question}

\title{Convolution comparison measures}
\author{Otte Heinävaara}
\address{The Division of Physics, Mathematics and Astronomy, California Institute of Technology, Pasadena, CA 91125}
\email{oeh@caltech.edu}
\date{February 10, 2026}

\begin{document}

\begin{abstract}
    We give a precise functional comparison between classical and free convolutions. If $\mu$ and $\nu$ are compactly supported probability measures, we show that the expectation of $f$ over the classical convolution $\mu * \nu$ is at least the expectation of $f$ over the free convolution $\mu \boxplus \nu$, as long as the fourth derivative of $f$ is non-negative. Conversely, the non-negativity of the fourth derivative is necessary for such a comparison.

    This comparison is based on the positivity of a related measure on $\R^{2}$, which we dub the \textit{convolution comparison measure}. We give an expression for this measure using a curious identity involving Hermitian matrices.
\end{abstract}

\maketitle

\section{Introduction}

\subsection{Comparing classical and free convolutions}

\sloppy Free independence was introduced by Voiculescu \cite{MR799593} in 1983 as an analogue of (classical) independence in tensor products, but for free products.

Concretely, if $\mathcal{A}_{1}, \mathcal{A}_{2}$ are unital $C^{*}$-algebras with states $\tau_{1}, \tau_{2}$, and $a_i \in \mathcal{A}_i$ are self-adjoint, we can think of $a_i$ as a compactly-supported real-valued random variable, with moments $\tau(a_{i}^{k})$. With the state $\tau := \tau_{1} \otimes \tau_{2}$ in the tensor product $\mathcal{A}_{1} \otimes \mathcal{A}_{2}$, the random variables become independent in the classical sense. In particular, the law of the element $a_1 + a_2 := a_{1} \otimes 1 + 1 \otimes a_{2}$ depends only on the laws of $a_{1}$ and $a_{2}$; it equals their (classical) convolution.

Voiculescu observed that there exists a state $\tau$ in the free product $\mathcal{A}_{1} * \mathcal{A}_{2}$ with the following property: if $a_{i}^{(j)} \in \mathcal{A}_{i}$ is such that $\tau_{i}\left(a_{i}^{(j)}\right) = 0$ for all $i \in \{1, 2\}$ and $j \in \{1, \dots, n\}$, then
\begin{align}\label{freenes_condition}
    \tau\left(a_{1}^{(1)} a_{2}^{(1)} a_{1}^{(2)} a_{2}^{(2)} \cdots a_{1}^{(n)} a_{2}^{(n)}\right) = 0.
\end{align}
This guides the way to a definition of the free independence of subalgebras of a \mbox{$C^{*}$-algebra}. If $a_{1}$ and $a_{2}$ are self-adjoint and freely independent in this sense, Voiculescu \cite{MR839105} showed that the law of $a_{1} + a_{2}$ again depends only on the laws of $a_{1}$ and $a_{2}$; the law of $a_1 + a_2$ is the free convolution.

Let $\mu$ and $\nu$ be compactly-supported real-valued random variables. Calculation of the classical convolution is often a straightforward task: if $\mu$ and $\nu$ are absolutely continuous w.r.t. the Lebesgue measure, then so is their classical convolution $\mu * \nu$, and its density is the convolution of the respective densities. If $\mu$ and $\nu$ are finitely supported, then so is their convolution; and the support and probabilities of the convolution admit a simple description. (See \cite[2.1.2]{MR3930614} for related computations.)

For the free convolution $\mu \boxplus \nu$, the story is considerably more complicated. Voiculescu studied the free convolution with the help of the $R$-transform \cite{MR839105}. The $R$-transform analytically captures freeness, much like characteristic functions do in the study of commutative random variables. Apart from some special cases \cite[Chapter 3]{MR3585560}, however, the $R$-transform avoids having an explicit expression, and so does the free convolution. For example, there is no simple formula for the free convolution of finitely-supported probability measures, unless both measures have at most two points in their support. Bercovici and Voiculescu \cite[Theorem 7.4]{MR1639647} showed that, in general, the free convolution of finitely-supported measures is not finitely supported, nor does it have any point masses.

In the classical setting, the support of $\mu * \nu$ is the Minkowski sum of the supports of the summands. Hence, the convex hull $\conv(\supp(\mu * \nu))$ also equals $\conv(\supp(\mu)) + \conv(\supp(\nu))$; this fact also follows from the $C^{*}$-algebra description of the convolution.

For the free convolution, it still follows from operator-theoretic considerations that $\conv(\supp(\mu \boxplus \nu))$ is a subset of $\conv(\supp(\mu)) + \conv(\supp(\nu))$; however, the inclusion is strict in general. In this sense, we can say that the free convolution is less spread out than classical convolution.

Moments of the free convolution $\mu \boxplus \nu$---while a lot more involved than in the classical setting---do admit a combinatorial description in terms of the moments of $\mu$ and $\nu$. Speicher \cite{MR1268597} observed that the moments of the free convolution can be expressed in terms of the so-called free cumulants. The resulting expressions, however, become increasingly complex for higher moments, and their manipulation becomes untenable.

Let $m_{k}(\mu)$ be the $k$th moment of $\mu$. It is straightforward to verify that the first three moments of the classical and free convolutions agree, but they differ from the fourth moment onwards in general. In fact, $m_{4}(\mu * \nu)$ is (almost) always strictly bigger than $m_{4}(\mu \boxplus \nu)$; the difference between the two is $2 \Var(\mu) \Var(\nu)$, where $\Var$ denotes the variance. Similarly,
\begin{align*}
    \limsup_{k \to \infty} m_{2 k}(\mu * \nu)^{1/(2 k)} \geq \limsup_{k \to \infty} m_{2 k}(\mu \boxplus \nu)^{1/(2 k)};
\end{align*}
this is a reformulation of the claim on the inclusion of the convex hulls of the supports.

It is a natural question to ask whether a similar ordering is true for higher even moments.

\subsection{Convolution comparison measures}

Our first result gives a precise ordering of classical and free convolutions, which, in particular, immediately answers our question on ordering the moments in the affirmative.

\begin{thm}\label{thm1}
    Let $\mu$ and $\nu$ be compactly-supported probability measures on $\R$. Then, for any $f \in C^{4}(\R)$ with $f^{(4)} \geq 0$, we have
    \begin{align}\label{main_ineq}
        \int f(t) \d{(\mu * \nu)}(t) \geq \int f(t) \d{(\mu \boxplus \nu)}(t),
    \end{align}
    where $*$ and $\boxplus$ denote the classical and free convolutions respectively. Conversely, if $f \in C^{4}(\R)$ is such that (\ref{main_ineq}) holds for any $\mu$ and $\nu$, then the fourth derivative of $f$ is non-negative.
\end{thm}

In particular, for any $k \in \mathbb{N}$,
\begin{align}\label{moment_comparison}
    m_{2 k}(\mu * \nu) \geq m_{2 k}(\mu \boxplus \nu).
\end{align}

Theorem~\ref{thm1} can be superficially restated to be about the positivity of an auxiliary measure on $\R$ derived from $\mu$ and $\nu$---indeed, the measure arising from the linear functional sending $f^{(4)}$ to ``left-hand side - right-hand side'' of (\ref{main_ineq}). The key insight of this paper is that we can extract this positivity from a measure on $\R^{2}$. This positive measure, which we'll call the \textit{convolution comparison measure} of $\mu$ and $\nu$, is supported on $\conv(\supp(\mu)) \times \conv(\supp(\nu))$. Theorem~\ref{thm1} can be deduced from the positivity of the slices of this convolution comparison measure.

\begin{thm}\label{thm2}
    Let $\mu$ and $\nu$ be compactly-supported probability measures on $\R$. Then, there exists a positive Radon measure $\widetilde{m}_{\mu, \nu}$ that is absolutely continuous w.r.t. the Lebesgue measure $m_{2}$ and is supported on  $\conv(\supp(\mu)) \times \conv(\supp(\nu)) \subset \R^{2}$, such that, for any $a, b \in \R \setminus \{0\}$ and $f \in C^{4}(\R)$, we have
    \begin{align}\label{general_identity}
        \widetilde{m}_{\mu, \nu}((x, y) \mapsto f^{(4)}(a x + b y)) = \frac{1}{a^{2} b^{2}}\left(\int f(t) d (a \mu * b \nu)(t) - \int f(t) d (a\mu \boxplus b\nu)\right).
    \end{align}
    $\widetilde{m}_{\mu, \nu}$ is the \emph{convolution comparison measure} of $\mu$ and $\nu$.
\end{thm}

Here, $a \mu$ is the probability measure satisfying $(a \mu)(a U) = \mu(U)$ for $U \subseteq \R$.

Using (\ref{general_identity}), we can check that $\widetilde{m}_{\mu, \nu}(\R^{2}) = 1/12 \Var(\mu) \Var(\nu)$; so, the measure vanishes if and only if one of the probability measures is constant.

One could give slightly generalized statements for Theorems~\ref{thm1} and \ref{thm2} by noting that $f$ need not be defined outside a (suitable) compact set. This generalization is superficial, however, as such $C^{4}$ functions can be extended to all of $\R$ anyhow.

Theorem~\ref{thm1} was previously established in the special case of $f(t) = -\log|z - t|$ (for any parameter $z$ lying outside $\conv(\supp(\mu)) + \conv(\supp(\nu))$) by Arizmendi and Johnston \cite{arizmendi2023free}. Arizmendi and Johnston derive their result as a by-product of the study of entropic optimal transport; they further rely on connections between the free convolution and quadrature formulas from the theory of random matrices \cite{MR3892446}. It is unclear if their approach can be generalized to prove Theorem~\ref{thm1}.

As a corollary of Theorem~\ref{thm2}, we also obtain the following concrete comparison criterion. For a bivariate real polynomial $p(x, y) = \sum_{i, j} c_{i, j} x^{i} y^{j}$, consider the symmetric non-commutative lift
\begin{align*}
    \widetilde{p}(a, b) = \sum_{i, j} c_{i, j} m_{i, j}(a, b),
\end{align*}
where $m_{i, j}(a, b)$ is the average of all the non-commutative reorderings of $a^{i} b^{j}$. Equivalently, the $m_{i,j}(a, b)$ are defined implicitly by the relation
\begin{align*}
    (t a + b)^{n} = \sum_{k = 0}^{n} \binom{n}{k} t^{k} m_{k, n - k}(a, b),
\end{align*}
when $a$ and $b$ are non-commuting, and $t$ is a scalar.

\begin{kor}\label{cor1}
    Let $p(x, y) = \sum_{i, j} c_{i, j} x^{i} y^{j}$ be a bivariate real polynomial. The following are equivalent:
    \begin{enumerate}[(a)]
        \item \label{cond1}The expectation of $\widetilde{p}(a, b)$ for independent $a$ and $b$ is at least the expectation of $\widetilde{p}(a, b)$ for freely independent $a$ and $b$, where $a$ and $b$ are bounded self-adjoint.
        \item \label{cond2}For any $x, y \in \R$,
        \begin{align*}
            \frac{\d{}^{2}}{\d{x}^{2}}\frac{\d{}^{2}}{\d{y}^{2}} p(x, y) = \sum_{i, j} c_{i, j} i (i - 1) j (j - 1) x^{i - 2} y^{j - 2} \geq 0.
        \end{align*}
    \end{enumerate}
\end{kor}

\begin{proof}
    We claim that
    \begin{align}\label{moment_difference}
        \mathbb{E}_{\substack{a \sim \mu, b \sim \nu \\ a, b \text{ are independent}}}\widetilde{p}(a, b) - \mathbb{E}_{\substack{a \sim \mu, b \sim \nu \\ a, b \text{ are freely} \\\text{
        independent}}}\widetilde{p}(a, b) = \widetilde{m}_{\mu, \nu}\left((x, y) \mapsto \frac{\d{}^{2}}{\d{x}^{2}}\frac{\d{}^{2}}{\d{y}^{2}}p(x, y)\right).
    \end{align}
    Indeed, by the properties of symmetric lifts, to show~(\ref{moment_difference}), it suffices to consider the polynomials $(x, y) \mapsto (t x + y)^{n}$. For these polynomials, the left-hand side  simplifies using the definition of $\widetilde{m}_{\mu, \nu}$ (\ref{general_identity}) to
    \begin{align*}
        t^{2} \widetilde{m}_{\mu, \nu}\left((x, y) \mapsto n (n - 1) (n - 2) (n - 3)(t x + y)^{n - 4}\right) = \widetilde{m}_{\mu, \nu}\left((x, y) \mapsto \frac{\d{}^{2}}{\d{x}^{2}}\frac{\d{}^{2}}{\d{y}^{2}}(t x + y)^{n}\right).
    \end{align*}

    Given (\ref{moment_difference}), positivity of $\widetilde{m}_{\mu, \nu}$ yields (\ref{cond2}) $\implies$ (\ref{cond1}).
    
    When $\mu$ and $\nu$ are supported near $x$ and $y$, $\widetilde{m}_{\mu, \nu}$ is supported near $(x, y)$. The right-hand side of~(\ref{moment_difference}) is then approximately $\d{}^2\d{}^2p(x,y)/\d{y^2}\d{x^2}$ by continuity, and thus, (\ref{cond1}) $\implies$ (\ref{cond2}).\end{proof}

\subsection{Commutation measuring measure}

The convolution comparison measure $\widetilde{m}_{\mu, \nu}$ admits an explicit expression for finitely supported $\mu$ and $\nu$. This expression is based on a peculiar linear algebra identity, inspired by the author's earlier work on tracial joint spectral measures \cite{MR4850603}.

For Hermitian $A, B \in M_{d}(\C)$, we consider the measure given by the non-negative function
\begin{align}\label{first_omega}
    (a, b) \mapsto \frac{1}{2 \pi}\sum_{i = 1}^{d} \left|\Im(\lambda_{i}((A - a I) (B - b I)))\right|.
\end{align}

We shall see that this measure vanishes if and only if $A$ and $B$ commute. In Lemma~\ref{weird_moment_lemma}, we provide a combinatorial expression for the moments of this measure. For instance,
\begin{align*}
    \int \frac{1}{2 \pi}\sum_{i = 1}^{d} \left|\Im(\lambda_{i}((A - a I) (B - b I)))\right| \d{a} \d{b} &= \frac{1}{3} \left(\tr(A^2 B^2) - \tr(A B A B)\right) \\
       &= \frac{1}{6}\tr((A B - B A)^{*} (A B - B A)).
\end{align*}
Starting with a finitely supported measure $\mu$, we can choose $A$ and $B$ appropriately such that the moments of~(\ref{first_omega}) can be expressed using the moments of the measure $\mu$; this will be the application in our setting. The positivity of (\ref{first_omega}) is the key ingredient in showing the positivity of the convolution comparison measure. We believe that further investigation of function (\ref{first_omega}) and the measure it defines is of independent interest.

\subsection{Overview of the proof and the paper}

As we mentioned, we will obtain an explicit expression for the convolution comparison measure $\widetilde{m}_{\mu, \nu}$; proving Theorem~\ref{thm2} using this expression amounts to proving an identity. Here are the main steps:

\begin{enumerate}
    \item In Proposition \ref{prop1}, we start with the right-hand side of the defining identity (\ref{general_identity}) with $f(t) = t^{n}$, and find a way to express it in terms of the moments of $\mu$ and $\nu$. This leads to an expression for $\widetilde{m}_{\mu, \nu}((t_{\mu}, t_{\nu}) \mapsto t_{\mu}^{n_{\mu}} t_{\nu}^{n_{\nu}})$ for $n_{\nu}, n_{\mu} \in \mathbb{N}$ in (\ref{basic_ccm_series}) in terms of certain linear functionals $I_{l}^{\cdot}$.
    \item In the case of finitely supported $\mu$ and $\nu$, we find a way to express the functionals $I_{l}^{\cdot}$ using Gegenbauer polynomials and the density (\ref{first_omega}) for appropriately chosen $A$ and $B$.
    \item In Proposition~\ref{prop3}, we evaluate a series involving Gegenbauer polynomials to obtain an evidently positive expression for the convolution comparison measure.
    \item A little approximation tells us that we can go beyond polynomial $f$ and finitely supported $\mu$ and $\nu$. This proves Theorem~\ref{thm2}. Theorem~\ref{thm1} follows easily.
\end{enumerate}

In section~\ref{prelim}, we give a detailed overview of the tools we need, and prove some of the identities we rely upon. In section~\ref{necessity}, we give a short but somewhat sketchy argument for the necessity of the fourth derivative condition, and argue why the condition is sufficient (at least to the leading order). Finally, in section~\ref{main_section} we prove Theorems~\ref{thm1} and \ref{thm2}. Section~\ref{discussion} is reserved for further comments and questions.

\section{Preliminaries and some identities}\label{prelim}

\subsection{Divided differences}

The divided differences of order $k$ are defined recursively as follows: let
\begin{align*}
	[x_{0}]_{f} = f(x_{0}),
\end{align*}
and for any positive integer $k$, recursively define
\begin{align}\label{lol}
	[x_{0}, x_{1}, \ldots, x_{k}]_{f} = \frac{[x_{0}, x_{1}, \ldots, x_{k - 1}]_{f} - [x_{1}, x_{2}, \ldots, x_{k}]_{f}}{x_{0} - x_{k}},
\end{align}
for pairwise distinct points $x_{0}, x_{1}, \ldots, x_{k}$. A short calculation yields an explicit expression
\begin{align}\label{lol2}
    [x_{0}, x_{1}, \ldots, x_{k}]_{f} = \sum_{i = 0}^{k} \frac{f(x_{i})}{\prod_{j \neq i} (x_{i} - x_{j})}.
\end{align}

If $f$ is $C^{k}$, the divided differences of order $k$ has a continuous extension to all tuples of $k + 1$~points. This extension satisfies (\ref{lol}) whenever $x_{0} \neq x_{k}$; and a mean value theorem
\begin{align}\label{dd_mean_value_theorem}
	[x_{0}, x_{1}, \ldots, x_{k}]_{f} = \frac{f^{(k)}(x)}{k!},
\end{align}
for some $x \in \conv \{x_{0}, x_{1}, \ldots, x_{k}\}$.

We shall employ the following expression for the divided differences of power functions
\begin{align}\label{dd_powers}
    [x_{0}, x_{1}, x_{2}, \ldots, x_{k}]_{t \mapsto t^{n}} = \sum_{\substack{i_{0}, i_{1}, \ldots, i_{m} \geq 0 \\ i_{0} + i_{1} + \ldots + i_{k} = n - k}} x_{0}^{i_{0}} x_{1}^{i_{1}} \cdots x_{k}^{i_{k}}.
\end{align}

For this, and a lot more, see the survey of de Boor \cite{MR2221566}.

\subsection{Hermitian matrix identities}

We will need some identities from the theory of tracial joint spectral measures \cite{MR4850603}.

We will denote the eigenvalues of $A \in M_{d}(\C)$ by $\lambda(A) = \{\lambda_{1}(A), \lambda_{2}(A), \ldots, \lambda_{d}(A)\}$, repeated with the algebraic multiplicities; the ordering will play no role for us.

\begin{thm}\label{tjsm}
    Let $d$ be a positive integer and $A, B \in M_{d}(\C)$ be Hermitian. Then, there exists a unique positive measure $\mu_{A, B}$ on $\R^{2}$ such that the following is true:

    Fix any measurable function $f$ on $\R$ such that $f(0) = 0$, and for any $M > 0$,
    \begin{align*}
        \int_{-M}^{M} \left|\frac{f(t)}{t}\right| \d{t} < \infty.
    \end{align*}
    Define a function $H(f) : \R \to \R$ by
    \begin{align*}
        H(f)(x) &= \int_{0}^{1} \frac{1 - t}{t} f(x t) \d{t}.
    \end{align*}
    Then, for any $x, y \in \R$, we have
    \begin{align}\label{mainIdentity}
        \tr H(f)(x A + y B) = \int_{\R^{2}} f(a x + b y) \d{\mu}_{A, B}(a, b).
    \end{align}
    Denote by $\mu_{c} = \mu_{c, A, B}$ and $\mu_{s} = \mu_{s, A, B}$ the continuous and singular parts of $\mu_{A, B}$ w.r.t. the Lebesgue measure $m_{2}$ on $\R^{2}$. We assume some linear combination of $A$ and $B$ is invertible. Then, the continuous part $\mu_{c}$ is given by
    \begin{align}\label{density_formula}
        \rho_{A, B}(a, b) := \frac{\d{\mu}_{c}}{\d{m}_{2}}(a, b) = \frac{1}{2 \pi}\sum_{i = 1}^{d} \left|\Im\left( \lambda_{i}\left(\left(I - \frac{a A + b B}{a^{2} + b^{2}}\right) (b A - a B)^{-1} \right)\right) \right|.
    \end{align}
    Furthermore, if $A$ is invertible and $A^{-1} B$ has $d$ distinct eigenvalues, the singular part $\mu_{s}$ satisfies
    \begin{align}\label{singular_formula}
        \mu_{s}(\varphi) = \sum_{v \in E(A^{-1} B)} \int_{0}^{1} \frac{1 - t}{t} \varphi\left( \frac{\langle A v, v \rangle}{\langle v, v\rangle} t, \frac{\langle B v, v\rangle}{\langle v, v\rangle} t\right) \d{t}.
    \end{align}
    where $E(C)$ denotes a set of eigenvectors of a matrix $C\in M_{d}(\C)$ and $\varphi$ is a smooth function with compact support that does not contain $0$.
\end{thm}

\begin{lem}\label{derivative_formula_lemma_2}
    Let $A, B \in M_{d}(\C)$ be Hermitian, with $A$ invertible, and $k \geq 2$. Then, the function $t \mapsto \tr (A + t B)_{+}^{k - 1}$ is smooth at $t = 0$, and
    \begin{align}\label{derivative_formula_2}
        \frac{\d{^{k}}}{\d {s^{k}}}\tr (A + s B)_{+}^{k - 1}\Big|_{s = 0} = (k - 1) k! \int_{\R} b^{k} \rho_{A, B}(0, b) \d{b},
    \end{align}
    where $\rho_{A, B}$ is as in (\ref{density_formula}).
\end{lem}

\begin{proof}
    We apply Theorem \ref{tjsm} with $f: t \mapsto t_{+}^{k - 1}$. We can check that $H(f) = f/(k (k - 1))$, and (\ref{mainIdentity}) produces
    \begin{align*}
        \tr (A + s B)_{+}^{k - 1} = k (k - 1) \int_{\R^{2}} (a + b s)_{+}^{k - 1} \d{\mu}_{A, B}(a, b).
    \end{align*}
    For a fixed $\delta > 0$, we can partition $\R^{2}$ into a wedge $S_{\delta} = \{(a, b) \in \R^{2} \mid |a| \leq \delta |b| \}$ and its complement. Observe that
    \begin{align*}
        \int_{\R^{2} \setminus S_{\delta}} (a + b s)_{+}^{k - 1} \d{\mu}_{A, B}(a, b)
    \end{align*}
    is a polynomial in $s$ of degree at most $k - 1$ for $|s| \leq |\delta|$, and hence its $k$th derivative vanishes.

    By~(\ref{singular_formula}), $\mu_{s}(S_{\delta}) = 0$ whenever $S_{\delta}$ doesn't contain any non-zero points of the form $(\langle A v, v\rangle/\langle v, v\rangle, \langle B v, v\rangle/\langle v, v\rangle)$, for an eigenvector $v$ of $A^{-1} B$ (with eigenvalue $\lambda$). Note that as $\lambda \langle A v, v\rangle = \langle B v, v\rangle$, non-zero points in $S_{\delta}$ correspond to  $1 \leq |\lambda| \delta$. By picking $\delta$ to be small enough, we can ensure that $\mu_{s}(S_{\delta}) = 0$.
    
    By~(\ref{density_formula}) \cite[Remark~3.2]{MR4850603}, the continuous part can be rewritten as
    \begin{align*}
        \rho_{A, B}(a, b) = \frac{1}{2 \pi b^2}\sum_{i = 1}^{d} \left|\Im\left( \lambda_{i}\left(\left(b I - B\right) \left(A - \frac{a}{b}B\right)^{-1} \right)\right) \right|.
    \end{align*}
    
    Using standard eigenvalue regularity results \cite{MR1477662}, we see that $b^2 \rho_{A, B}(a, b)$ is bounded on $S_{\delta}$ and continuous on $S_{\delta} \setminus \{(0, 0)\}$. Hence, the left-hand side of~(\ref{derivative_formula_2}) depends only on the function $\rho(0, b)$. For the exact dependence, by approximation, it suffices to calculate the $k$th derivative of
    \begin{align*}
        k (k - 1) \int_{b_{1}}^{b_{2}} \int_{-|b|}^{|b|} (a + b s)_{+}^{k - 1} b^{l - 2} \d{a} \d{b}
    \end{align*}
    at $s = 0$, and confirm that it agrees with
    \begin{align*}
        (k - 1) k! \int_{b_{1}}^{b_{2}} b^{k + l - 2} \d{b}.
    \end{align*}

    This amounts to computing
    \begin{align*}
        k (k - 1)\int_{b_{1}}^{b_{2}} \int_{-|b|}^{|b|} (a + b s)_{+}^{k - 1} b^{l - 2} \d{a} \d{b} &= (k - 1)\int_{b_{1}}^{b_{2}} |b|^{k} \left(s \sign(b) + 1\right)^{k} b^{l - 2} \d{b}\\ &= s^{k} (k - 1)\int_{b_{1}}^{b_{2}}  b^{k + l - 2} \d{b} + p(s),
    \end{align*}
    where $p$ is a polynomial in $s$ of degree less than $k$. This last expression evidently has the right $k$th derivative.
\end{proof}

\begin{lem}\label{weird_moment_lemma}
    Let $A, B \in M_{d}(\C)$ be Hermitian. Define a function $\omega_{A, B}$ as
    \begin{align*}
        \omega_{A, B}(a, b) := \frac{1}{2 \pi}\sum_{i = 1}^{d} \left|\Im(\lambda_{i}((A - a I) (B - b I)))\right|.
    \end{align*}

    Then, $\omega_{A, B}$ is continuous and compactly supported on $\conv(\lambda(A)) \times \conv(\lambda(B))$, and for any $k, l \in \mathbb{N}$, we have
    \begin{align}\label{moment_expression}
        \int a^{k} b^{l} \omega_{A, B}(a, b) \d{a} \d{b} = \frac{1}{(k + l + 2) (k + l + 3)} \sum_{n = 1}^{\min(k, l) + 2} \frac{(-1)^{n - 1}}{\binom{k + l + 1}{n - 1}} \sum_{\substack{i_{1}, \ldots, i_{n}, j_{1}, \ldots, j_{n} \geq 1 \\ i_{1} + \ldots, i_{n} = k + 2 \\ j_{1} + \ldots, j_{n} = l + 2 \\}} \tr(A^{i_{1}} B^{j_{1}} \cdots A^{i_{n}} B^{j_{n}}).
    \end{align}
\end{lem}

\begin{proof}
    Continuity of $\omega_{A,B}$ follows from continuity of eigenvalues \cite{MR1477662}, while the claim on its support follows from \cite[Theorem 10.1]{MR2178635}.
    
    For $a \notin \lambda(A)$ consider the tracial joint spectral measure $\mu_{A', B'}$ for the pair ${(A', B') = ((A - a I)^{-1}, B)}$. As in the proof of Lemma~\ref{derivative_formula_lemma_2}, $\rho_{A', B'}$ satisfies
    \begin{align*}
        \rho_{A', B'}(0, b) = \frac{1}{2 \pi}\sum_{i = 1}^{d} \left|\Im\left( \lambda_{i}\left(\left(I - \frac{B'}{b}\right) (b A')^{-1} \right)\right) \right| &= \frac{1}{2 \pi b^{2}}\sum_{i = 1}^{d} \left|\Im\left( \lambda_{i}\left((A - a I) (B - b I)\right)\right) \right| \\
        &= \frac{\omega_{A, B}(a, b)}{b^{2}}.
    \end{align*}
    Lemma~\ref{derivative_formula_lemma_2} thus implies that for any $l \in \mathbb{N}$,
    \begin{align}\label{weird_moment_lemma_id_1}
        \int b^{l} \omega_{A, B}(a, b) \d{b} = \frac{1}{(l + 1)(l + 2)!}\frac{\d{^{l + 2}}}{\d {s^{l + 2}}}\tr ((A - a I)^{-1} + s B)_{+}^{l + 1}\Big|_{s = 0}.
    \end{align}
    Using \cite[Lemma 2.2.3]{MR4736654}, the expression on the right-hand side can be rewritten using divided differences as
    \begin{align}\label{weird_moment_lemma_id_2}
        \frac{1}{l + 2}\sum_{i_{1} = 1}^{d} \sum_{i_{2} = 1}^{d} \cdots \sum_{i_{l + 2} = 1}^{d} [(\lambda_{i_{1}} - a)^{-1}, (\lambda_{i_{2}} - a)^{-1}, \ldots, (\lambda_{i_{l + 2}} - a)^{-1}]_{t \mapsto t_{+}^{l}} B_{i_{1}, i_{2}} B_{i_{2}, i_{3}} \cdots B_{i_{k + 2}, i_{1}},
    \end{align}
    where $\lambda_{1}, \lambda_{2}, \ldots, \lambda_{d}$ are the eigenvalues of $A$, and $(B_{i, j})_{i, j}^{d}$ is the matrix of $B$ in the respective eigenbasis.

    Assume for now that the $\lambda_{i_{j}}$s are pairwise distinct. Then, using the explicit expression (\ref{lol2}) for divided differences, we can write
    \begin{align}\label{manipulation}
        \nonumber[(\lambda_{i_{1}} - a)^{-1}, \ldots, (\lambda_{i_{l + 2}} - a)^{-1}]_{t \mapsto t_{+}^{l}} &= \sum_{j : \lambda_{i_{j}} > a} \frac{(\lambda_{i_{j}} - a)^{-l}}{\prod_{j' \neq j} ((\lambda_{i_{j}} - a)^{-1} - (\lambda_{i_{j'}} - a)^{-l})} \\&= \sum_{j : \lambda_{i_{j}} > a} \frac{\prod_{j' = 1}^{l + 2} (\lambda_{i_{j'}} - a)}{\prod_{j' \neq j} (\lambda_{i_{j'}} - \lambda_{i_{j}})}.
    \end{align}

    If $F_{k}$ is an antiderivative of the function $a \mapsto \prod_{j' = 1}^{l + 2} (\lambda_{i_{j'}} - a) a^{k}$, we see that
    \begin{align*}
        \int \sum_{j : \lambda_{i_{j}} > a} \frac{\prod_{j' = 1}^{l + 2} (\lambda_{i_{j'}} - a)}{\prod_{j' \neq j} (\lambda_{i_{j'}} - \lambda_{i_{j}})} a^{k} \d{a} = \sum_{j = 1}^{l + 2} \frac{F_{k}(\lambda_{i_{j}})}{\prod_{j' \neq j} (\lambda_{i_{j'}} - \lambda_{i_{j}})} = (-1)^{l + 1}[\lambda_{i_{1}}, \lambda_{i_{2}}, \ldots, \lambda_{i_{l + 2}}]_{F_{k}}.
    \end{align*}
    Even if the $\lambda_{i_{j}}$s are not pairwise distinct, continuity of divided differences implies that
    \begin{align}\label{weird_moment_lemma_id_3}
         \int a^{k} [(\lambda_{i_{1}} - a)^{-1}, (\lambda_{i_{2}} - a)^{-1}, \ldots, (\lambda_{i_{l + 2}} - a)^{-1}]_{t \mapsto t_{+}^{l}}\d{a} = (-1)^{l + 1}[\lambda_{i_{1}}, \lambda_{i_{2}}, \ldots, \lambda_{i_{l + 2}}]_{F_{k}}.
    \end{align}
    
    Write $e_{j}$ for the degree $j$ elementary symmetric polynomial of $\lambda_{i_{1}}, \lambda_{i_{2}}, \ldots, \lambda_{i_{l + 2}}$. Then, using formula~(\ref{dd_powers}) on divided differences for the power function, we obtain
    \begin{align}\label{weird_moment_lemma_id_4}
        (-1)^{l + 1}[\lambda_{i_{1}}, \lambda_{i_{2}}, \ldots, \lambda_{i_{l + 2}}]_{F_{k}} = \sum_{j = 0}^{l + 2} \frac{(-1)^{j - 1} e_{j}}{l + k + 3 - j} \sum_{\substack{p_{1}, p_{2}, \ldots, p_{l + 2} \geq 0 \\ p_{1} + p_{2} + \ldots + p_{l + 2} = k - j + 2}} \lambda_{i_{1}}^{p_{1}} \lambda_{i_{2}}^{p_{2}} \cdots \lambda_{i_{l + 2}}^{p_{l + 2}}.
    \end{align}
    The sum on the right expands to
    \begin{align*}
        \sum_{\substack{p_{1}, p_{2}, \ldots, p_{l + 2} \geq 0 \\ p_{1} + p_{2} + \ldots + p_{l + 2} = k + 2}} c_{k, l}(p_{1}, p_{2}, \ldots, p_{l + 2})\lambda_{i_{1}}^{p_{1}} \lambda_{i_{2}}^{p_{2}} \cdots \lambda_{i_{l + 2}}^{p_{l + 2}},
    \end{align*}
    where $c_{k, l}$ depends only on the number of positive $p_i$s. If the number of positive $p_i$s is $n$,
    \begin{align}\label{weird_moment_lemma_id_5}
        c_{k, l}(p_{1}, p_{2}, \ldots, p_{l + 2}) = \sum_{j = 0}^{l + 2} \frac{(-1)^{j - 1}}{l + k + 3 - j} \binom{n}{j} = \frac{(-1)^{n - 1}}{(l + k + 3)\binom{l + k + 2}{n}}.
    \end{align}
    The last equality amounts to a special case of the Chu--Vandermonde identity (\ref{gauss_summation}); we discuss this, and related identities, in the next section.

    Using identities (\ref{weird_moment_lemma_id_1}), (\ref{weird_moment_lemma_id_2}), (\ref{weird_moment_lemma_id_3}), (\ref{weird_moment_lemma_id_4}), and (\ref{weird_moment_lemma_id_5}), we have
    \begin{align*}
        & \int a^{k} b^{l} \omega_{A, B}(a, b) \d{a}\d{b} \\
        =& \frac{1}{l + 2}\sum_{n = 1}^{\min(k, l) + 2} \frac{(-1)^{n - 1}}{(l + k + 3)\binom{l + k + 2}{n}} \sum_{i_{1} = 1}^{d}  \cdots \sum_{i_{l + 2} = 1}^{d} \sum_{\substack{p_{1}, p_{2}, \ldots, p_{l + 2} \geq 0 \\ p_{1} + p_{2} + \ldots + p_{l + 2} = k + 2 \\ \#\{j \in \{1, \dots, l + 2\}\mid p_{j} > 0\} = n}} \lambda_{i_{1}}^{p_{1}} B_{i_{1}, i_{2}} \cdots \lambda_{i_{k + 2}}^{p_{k + 2}}B_{i_{k + 2}, i_{1}} \\
        =& \frac{1}{l + 2}\sum_{n = 1}^{\min(k, l) + 2} \frac{(-1)^{n - 1}}{(l + k + 3)\binom{l + k + 2}{n}} \sum_{\substack{i_{1}, \ldots, i_{n}, j_{1}, \ldots, j_{n} \geq 1 \\ i_{1} + \ldots, i_{n} = k + 2 \\ j_{0} + j_{1} + \ldots, j_{n} = l + 2 \\}} \tr(B^{j_{0}} A^{i_{1}} B^{j_{1}} \cdots A^{i_{n}} B^{j_{n}}).
    \end{align*}
    It remains to observe that
    \begin{align*}
         \sum_{\substack{i_{1}, \ldots, i_{n}, j_{1}, \ldots, j_{n} \geq 1 \\ i_{1} + \ldots, i_{n} = k + 2 \\ j_{0} + j_{1} + \ldots, j_{n} = l + 2 \\}} \tr(B^{j_{0}} A^{i_{1}} B^{j_{1}} \cdots A^{i_{n}} B^{j_{n}}) =& \sum_{\substack{i_{1}, \ldots, i_{n}, j_{1}, \ldots, j_{n} \geq 1 \\ i_{1} + \ldots, i_{n} = k + 2 \\ j_{1} + \ldots, j_{n} = l + 2 \\}} j_{n}\tr(A^{i_{1}} B^{j_{1}} \cdots A^{i_{n}} B^{j_{n}}) \\
         =& \frac{l + 2}{n} \sum_{\substack{i_{1}, \ldots, i_{n}, j_{1}, \ldots, j_{n} \geq 1 \\ i_{1} + \ldots, i_{n} = k + 2 \\ j_{1} + \ldots, j_{n} = l + 2 \\}} \tr(A^{i_{1}} B^{j_{1}} \cdots A^{i_{n}} B^{j_{n}});
    \end{align*}
    after a little rearrangement, we obtain the desired result~(\ref{moment_expression}).
\end{proof}

\subsection{Hypergeometric identities and Gegenbauer polynomials}\label{hypergeometric}

We need various identities involving binomial coefficients and factorials; these follow from known properties of hypergeometric functions. See \cite{MR4824080} for a modern treatment of hypergeometric functions.

For $p, q \in \mathbb{N}$, and real numbers $a_{1}, \ldots, a_{p}, b_{1}, \ldots, b_{q}$, the generalized hypergeometric function ${}_pF_q$ is defined by the series
\begin{align}\label{hypergeometric_series}
    \pFq{p}{q}{a_{1}, \ldots, a_{p}}{b_{1}, \ldots, b_{q}}{z} = \sum_{k = 0}^{\infty} \frac{a_{1}^{(k)} \cdots a_{p}^{(k)}}{b_{1}^{(k)} \cdots b_{p}^{(k)}} \frac{z^{k}}{k!}.
\end{align}
We only need to evaluate ${}_pF_q$ when $(p, q) = (2, 1)$ or $(3, 2)$, and one of the $a_{i}$s is a negative integer; in this case, the defining series~(\ref{hypergeometric_series}) is a finite sum. Here, $a^{(k)}$ denotes the rising factorial $a^{(k)} = a (a + 1) \cdots (a + k - 1)$. We denote the falling factorial with $a_{(k)} = a (a - 1) \cdots (a - k + 1)$.

The Chu-Vandermonde identity \cite[4.5.1]{MR4824080} provides the evaluation in case $(p,q) = (2,1)$: if $n \in \mathbb{N}$,
\begin{align}\label{gauss_summation}
     \pFq{2}{1}{a, -n}{c}{1} = \frac{(c - a)^{(n)}}{c^{(n)}}.
\end{align}
Saalsch{\"u}tz's theorem \cite[4.8]{MR4824080} attacks the case of $(p,q) = (3,2)$: if $n \in \mathbb{N}$,
\begin{align}\label{saalschutz}
    \pFq{3}{2}{a, b, -n}{c, 1 + a + b - c - n}{1} = \frac{(c - a)^{(n)} (c - b)^{(n)}}{c^{(n)} (c - a - b)^{(n)}}.
\end{align}

\begin{lem}\label{identity_lemma}
    Let $n_{1}, n_{2} \in \mathbb{N}$. Then,
    \begin{align*}
        \sum_{l = 0}^{\infty} (2 l + 1) (-1)^{l} \frac{(n_{1})_{(l)}}{(n_{1} + l + 1)!} \frac{(n_{2})_{(l)}}{(n_{2} + l + 1)!} = \frac{1}{(n_{1} + n_{2} + 1)!}.
    \end{align*}
\end{lem}

\begin{proof}
     Without loss of generality, we assume that $n_{1} \leq n_{2}$. With some simple shuffling, the claim is equivalent to
    \begin{align*}
        \sum_{l=0}^{n_1} \overbrace{(2 n_{1} + 1 - 2 l) (-1)^{n_{1} - l} \frac{(n_{1} + n_{2} + 1)_{(2 n_{1} + 1 - l)} (n_{1} + n_{2} + 1)_{(l)}}{(l)! (2 n_{1} + 1 - l)!}}^{\text{denoted }q(l)} &= \frac{(n_{1} + n_{2} + 1)!}{n_{1} ! n_{2}!}.
    \end{align*}

    Notice that the sum extends palindromically until $2 n_{1} + 1$, in that $q(l) = q(2n_1 + 1- l)$. Summing from $l = 0$ to $l = 2n_1 + 1$ doubles the identity.

	By distributing the first factor, write $q(l) = (2n_1 + 1) (-1)^{n_1 - l} (\cdots) + (-2l) (-1)^{n_1 - l} (\cdots)$. This splits the sum into two; the first sum vanishes, as the $l$th term cancels the $(2 n_{1} + 1 - l)$th term. The second sum can be written in terms of
    \begin{align*}
        \pFq{2}{1}{-2 n_{1}, -n_{1} - n_{2}}{n_{2} - n_{1} + 2}{-1},
    \end{align*}
    for which we find an expression in \cite[Section 3]{MR1406215}.
\end{proof}

We write $C_{k}^{3/2}$ for the $k$th Gegenbauer polynomial with parameter $3/2$, normalized to equal $1$ at $1$. These are orthogonal polynomials on $[-1, 1]$ with weight $1 - x^{2}$. With our normalization, they satisfy \cite[9.2.2]{MR4824080}, for appropriate $x, t, k$, and $l$,
\begin{align}
    (1 - 2 x t + t^{2})^{-3/2} =& \sum_{k = 0}^{\infty} \frac{(k + 1) (k + 2)}{2} C_{k}^{3/2}(x) t^{k}, \label{gegenbauer_series} \\
    \int_{-1}^{1} C_{k}^{3/2}(x) C_{l}^{3/2}(x) (1 - x^{2}) \d{x} =& \delta_{k, l} \frac{8}{(2 k + 3) (k + 1) (k + 2)}, \label{gegenbauer_orthogonality} \\
    C_{k}^{3/2}(x) =& \pFq{2}{1}{k + 3, -k}{2}{\frac{1 - x}{2}}\label{gegenbauer_hypergeometric}.
\end{align}

\subsection{Free probability}

We provide a brief overview of free probablity; an interested reader may find a comprehensive account in \cite{MR3585560}.

We work with probability measures with compact support. The moments of such a measure $\mu$ are denoted by $m_{k}(\mu) := \mathbb{E}_{X \sim \mu} X^{k}$.

Central to the notion of free convolution are free cumulants, which we denote by $\kappa_{k}(\mu)$. If
\begin{align}
    \psi_{\mu}(z) = z \left(1 + \sum_{k = 1}^{\infty} m_{k}(\mu) z^{k}\right) \label{moment_generator}
\end{align}
then the free cumulants of $\mu$ can be extracted from the formal inverse of $\psi_{\mu}(z)$ in $z$, with
\begin{align*}
    \psi_{\mu}^{-1}(z) = \frac{z}{1 + \sum_{k = 1}^{\infty} \kappa_{k}(\mu) z^{k}}.
\end{align*}

Fundamental properties of free cumulants are \cite{MR1268597,MR839105}
\begin{align}
    \kappa_{k}(a \mu) &= a^{k} \kappa_{k}(\mu), \label{cumulant_rule_1}\\
    \kappa_{k}(\mu \boxplus \nu) &= \kappa_{k}(\mu) + \kappa_{k}(\nu). \label{cumulant_rule_2}
\end{align}

Moments and free cumulants are related by the following identities.

\begin{lem}\label{cumulant_moment_identities}
    Let $\mu$ be a compactly-supported probability measure. For $n \geq k \geq 1$, let
    \begin{align*}
        M_{n, k}(\mu)~&= \sum_{\substack{i_{1}, i_{2}, \ldots, i_{k} \geq 1 \\ i_{1} + i_{2} + \ldots + i_{k} = n}} m_{i_{1}}(\mu) m_{i_{2}}(\mu) \cdots m_{i_{k}}(\mu), \\
        K_{n, k}(\mu)~&= \sum_{\substack{i_{1}, i_{2}, \ldots, i_{k} \geq 1 \\ i_{1} + i_{2} + \ldots + i_{k} = n}} \kappa_{i_{1}}(\mu) \kappa_{i_{2}}(\mu) \cdots \kappa_{i_{k}}(\mu).
    \end{align*}

    Then, for $n \geq k \geq 1$, we have the identities
    \begin{align}
        M_{n, k}(\mu)~&= \sum_{r = k}^{n} \frac{k}{r} \binom{n}{r - k} K_{n, r}(\mu), \label{mk_identity}\\
        K_{n, k}(\mu)~&= \sum_{r = k}^{n} (-1)^{k - r} \frac{k}{r}\binom{n + r - k - 1}{r - k} M_{n, r}(\mu). \label{km_identity}
    \end{align}
\end{lem}

Our argument to prove Lemma~\ref{cumulant_moment_identities} will be a slight generalization of Mottelson's \cite{mottelson2012introduction}, which further relies on the Lagrange inversion theorem \cite[Theorem 5.4.2]{MR4621625}. Denote by $[z^k] f(z)$ the coefficient of $z^k$ in $f(z)$.
\begin{thm}\label{lagrange_inversion}
    Let $f(z)$ be a formal power series with complex coefficients, with $[z^{0}]f(z) = 0$ and $[z^{1}]f(z) \neq 0$. Then, $f(z)$ has a formal inverse $g(z)$, and for any $k, n \in \Z$,
    \begin{align}\label{lagrange_inversion_identity}
        k[z^{k}] g^{n} = n [z^{-n}] f^{-k}.
    \end{align}
\end{thm}

\begin{proof}[Proof of Lemma \ref{cumulant_moment_identities}]
    Let $\psi_{\mu}$ be the power series as in (\ref{moment_generator}).

    Expressing the moments using $\psi_{\mu}(z)$, we obtain
    \begin{align*}
        & M_{n, k}(\mu) = [z^{n}] \sum_{n = k}^{\infty} M_{n, k}(\mu) z^{n} = [z^{n}] \left(\frac{\psi_{\mu}(z)}{z} - 1\right)^{k} = \sum_{l = 1}^{k} (-1)^{k - l} \binom{k}{l} [z^{n + l}]\left(\psi_{\mu}(z)^{l}\right) \\
        &\quad= \sum_{l = 1}^{k} (-1)^{k - l} \binom{k}{l} \frac{l}{n + l} [z^{-l}] (\psi_{\mu}^{-1}(z))^{-(n + l)} = \sum_{l = 1}^{k} (-1)^{k - l} \binom{k}{l} \frac{l}{n + l} [z^{n}] \left(1 + \sum_{j = 1}^{\infty} \kappa_{j}(\mu) z^{j}\right)^{n + l} \\
        &\quad= \sum_{l = 1}^{k} (-1)^{k - l} \binom{k}{l} \frac{l}{n + l} \sum_{r = 1}^{n} \binom{n + l}{r} K_{n, r}(\mu).
    \end{align*}

    It remains to verify that
    \begin{align*}
        \sum_{l = 1}^{k} (-1)^{k - l} \binom{k}{l} \frac{l}{n + l} \binom{n + l}{r} = \frac{k}{r} \binom{n}{r - k},
    \end{align*}
    which boils down to the evaluation of $\pFq{2}{1}{1 - k, n + 1}{n - r + 2}{1}$.

    This proves~(\ref{mk_identity}). The second identity~(\ref{km_identity}) can be proven similarly; or can be proven by verifying that the two linear systems (\ref{mk_identity}) and (\ref{km_identity}) are inverses of each other; this, too, boils down to a hypergeometric identity.
\end{proof}

\section{Leading order calculation}\label{necessity}

Theorem~\ref{thm2} implies that $f^{(4)} \geq 0$ is necessary for the comparison of convolutions~(\ref{main_ineq}) in Theorem~\ref{thm1}. In the following proposition, we sketch a simpler independent argument for this necessity. The argument also shows that $f^{(4)} \geq 0$ is sufficient to the leading order.

\begin{prop}
    Let $\mu$ and $\nu$ be compactly-supported probability measures, and $f \in C^{4}(\R)$. Then,
    \begin{align}\label{leading_order_expansion}
        \int f(t) \d{(\mu * \varepsilon \nu)}(t) - \int f(t) \d{(\mu \boxplus \varepsilon \nu)}(t) = \varepsilon^{2} \Var(\nu) \int f^{(4)}(t) \d{\widetilde{\mu}}(t) + o(\varepsilon^{2}),
    \end{align}
    for a positive measure $\widetilde{\mu}$, supported on $\conv(\supp(\mu))$, that depends only on $\mu$.
\end{prop}

\begin{proof}[Proof sketch]
    Start with the left-hand side of~(\ref{leading_order_expansion}) when $f(t) = t^{k}$; that is, with the expression $m_{k}(\mu * \varepsilon \nu) - m_{k}(\mu \boxplus \varepsilon \nu)$. Moments of the classical convolution can be expanded with the binomial theorem. For the free convolution, we can use the main freeness property~(\ref{freenes_condition}) to express the moments of the free convolution in terms of the moments of $\mu$ and $\nu$, at least up to $o(\varepsilon^{2})$; see \cite[1.12]{MR3585560} for a primer on how to manipulate such expressions. We obtain
    \begin{align}\label{k_expansion}
        \int t^{k} \d{(\mu * \varepsilon \nu)}(t) - \int t^{k} \d{(\mu \boxplus \varepsilon \nu)}(t) = \frac{\varepsilon^{2}}{2} \Var(\nu) \left(k(k - 1) m_{k - 2} - k\sum_{\substack{i_{0}, i_{1}\geq 0 \\ i_{0} + i_{1} = k - 2}} m_{i_{0}} m_{i_{1}} \right) + o(\varepsilon^{2}),
    \end{align}
    where $m_{k} = m_{k}(\mu)$.

    With some simple manipulation, we see that
    \begin{align*}
       \left(k(k - 1) m_{k - 2} - k\sum_{\substack{i_{0}, i_{1}\geq 0 \\ i_{0} + i_{1} = k - 2}} m_{i_{0}} m_{i_{1}} \right) = \mathbb{E}_{X \sim \mu} f''(X) - \mathbb{E}_{\substack{X, X' \sim \mu \\ X, X' \text{iid}}} [X, X']_{f'} =: I_{\mu}\left(f^{(4)}\right),
    \end{align*}
    for $f(t) = t^{k}$; and by approximation, for any $f \in C^{4}(\R)$,
    \begin{align*}
        \int f(t) \d{(\mu * \varepsilon \nu)}(t) - \int f(t) \d{(\mu \boxplus \varepsilon \nu)}(t) = \frac{\varepsilon^{2}}{2} \Var(\nu) I_{\mu}\left(f^{(4)}\right) + o(\varepsilon^{2}).
    \end{align*}

	It remains to check that $I_\mu$ is a positive linear functional, cf.~(\ref{leading_order_expansion}). 
    Consider $f(t) = \frac{1}{6}t_{+}^{3}$. We have $f^{(4)} = \delta_{0}$, and hence
    \begin{align*}
        I_{\mu}(\delta_{0}) = \mathbb{E}_{X \sim \mu} X_{+} - \frac{1}{2}\mathbb{E}_{\substack{X, X' \sim \mu \\ X, X' \text{iid}}} \frac{X X_{+} - X' X'_{+}}{X - X'} = \frac{1}{2}\mathbb{E}_{\substack{X, X' \sim \mu \\ X, X' \text{iid}}} \frac{X X'_{+} - X' X_{+}}{X - X'}.
    \end{align*}
    The last expression inside the expectation is non-negative pointwise, and hence $I_{\mu}(\delta_{0}) \geq 0$. Similarly, $I_{\mu}(\delta_{c}) \geq 0$ for any $c \in \R$. This gives the positivity of $I_{\mu}$.
\end{proof}

\section{Proofs of Theorems~\ref{thm2} and~\ref{thm1}}\label{main_section}

We prove Theorem~\ref{thm2} in four main steps.

\subsection{Step 1}

We start by introducing auxiliary functionals, which we shall use to express the difference of classical and free convolutions. For $l \in \N$, define a functional $I_{l}^\mu$ by letting
\begin{align}\label{I_first_definition}
    I_l^\mu\left(t \mapsto t^{n}\right) = \frac{n!}{(n + l + 3)! (n - l)!}\sum_{k \geq 1}  \frac{(-1)^{k - 1} (n - k + 2)!}{k!} \frac{(l + k)!}{(l - k + 2)!}  M_{n + 2, k}(\mu),
\end{align}
and extending linearly; here, $M_{n,k}(\mu)$ depends on $\mu$ as in Lemma~\ref{cumulant_moment_identities}. Set $1/m!$ to be $0$ for negative $m$, so  that $I_{l}^\mu$ vanishes on polynomials of degree less than $l$. Note further that the sum for $I_l^\mu$ has $l + 2$ terms.

The $I_l^\cdot$ help us in our quest to define the \emph{convolution comparison measure} for measures $\mu$ and $\nu$; define $\widetilde{m}_{\mu, \nu}$ by letting
\begin{align}\label{basic_ccm_series}
    \widetilde{m}_{\mu, \nu}((t_{\mu}, t_{\nu}) \mapsto t_{\mu}^{n_{\mu}} t_{\nu}^{n_{\nu}}) = \sum_{l = 0}^{\infty} (-1)^{l} (2 l + 3) I_l^\mu\left(t \mapsto t^{n_{\mu}}\right) I_l^\nu\left(t \mapsto t^{n_{\nu}}\right),
\end{align}
and extending linearly to bivariate real polynomials. In the rest of the proof, we verify that $\widetilde{m}_{\mu, \nu}$ satisfies~(\ref{general_identity}) and present an alternate expression for it to verify that it is a positive measure.

\begin{prop}\label{prop1}
    Let $p$ be a polynomial, and $a, b \in \R \setminus \{0\}$. Then,
    \begin{align}\label{ccm_identity_alt2}
        \widetilde{m}_{\mu, \nu}((x, y) \mapsto p^{(4)}(a x + b y)) = \frac{(a \mu * b \nu)(p) - (a \mu \boxplus b \nu)(p)}{a^{2} b^{2}}.
    \end{align}
\end{prop}

\begin{proof}
    Firstly, by Lemma \ref{cumulant_moment_identities} one has
    \begin{align*}
        m_{n}(\mu * \nu) &= m_{n}(\mu) + m_{n}(\nu) + \sum_{\substack{n_{\mu}, n_{\nu} \geq 1 \\ n_{\mu} + n_{\nu} = n}} \binom{n}{n_{\mu}} m_{n_{\mu}}(\mu) m_{n_{\nu}}(\nu) \\
        &= m_{n}(\mu) + m_{n}(\nu) \\
        &\quad+ \sum_{k_{\mu}, k_{\nu} \geq 1} \sum_{\substack{n_{\mu}, n_{\nu} \geq 1 \\ n_{\mu} + n_{\nu} = n}} \frac{1}{(1 + n_{\mu}) (1 + n_{\nu})} \binom{n}{n_{\mu}} \binom{n_{\mu} + 1}{k_{\mu}} \binom{n_{\nu} + 1}{k_{\nu}} K_{n_{\mu}, k_{\mu}}(\mu) K_{n_{\nu}, k_{\nu}}(\nu).
    \end{align*}
    On the free side, the additivity of cumulants~(\ref{cumulant_rule_2}), along with some combinatorics, justifies for $n \geq k \geq 1$ that
    \begin{align*}
        K_{n, k}(\mu \boxplus \nu) = K_{n, k}(\mu) + K_{n, k}(\nu) + \sum_{\substack{k_{\mu}, k_{\nu} \geq 1 \\ k_{\mu} + k_{\nu} = k}} \sum_{\substack{n_{\mu}, n_{\nu} \geq 1 \\ n_{\mu} + n_{\nu} = n}} \binom{k}{k_{\mu}} K_{n_{\mu}, k_{\mu}}(\mu) K_{n_{\nu}, k_{\nu}}(\nu),
    \end{align*}
    which then allows us to calculate
    \begin{align*}
        m_{n}(\mu \boxplus \nu) &= m_{n}(\mu) + m_{n}(\nu) + \sum_{k = 1}^{n} \sum_{\substack{k_{\mu}, k_{\nu} \geq 1 \\ k_{\mu} + k_{\nu} = k}} \sum_{\substack{n_{\mu}, n_{\nu} \geq 1 \\ n_{\mu} + n_{\nu} = n}} \frac{1}{k} \binom{n}{k - 1} \binom{k}{k_{\mu}} K_{n_{\mu}, k_{\mu}}(\mu) K_{n_{\nu}, k_{\nu}}(\nu) \\
        &= m_{n}(\mu) + m_{n}(\nu) + \frac{1}{n + 1}\sum_{k_{\mu}, k_{\nu} \geq 1} \sum_{\substack{n_{\mu}, n_{\nu} \geq 1 \\ n_{\mu} + n_{\nu} = n}} \binom{n + 1}{k_{\mu} + k_{\nu}} \binom{k_{\mu} + k_{\nu}}{k_{\mu}} K_{n_{\mu}, k_{\mu}}(\mu) K_{n_{\nu}, k_{\nu}}(\nu).
    \end{align*}

    Subtracting the two expressions yields
    \begin{align*}
        & m_{n}(\mu * \nu) - m_{n}(\mu \boxplus \nu) \nonumber = \sum_{\substack{n_{\mu}, n_{\nu} \geq 1 \\ n_{\mu} + n_{\nu} = n}} \sum_{k_{\mu}, k_{\nu} \geq 1} \frac{n!}{k_{\mu}! k_{\nu}!} \\
        & \left(\frac{1}{((n_{\mu} - k_{\mu}) + 1)! ((n_{\nu} - k_{\nu}) + 1)!} - \frac{1}{((n_{\mu} - k_{\mu}) + (n_{\nu} - k_{\nu}) + 1)!}\right) K_{n_{\mu}, k_{\mu}}(\mu) K_{n_{\nu}, k_{\nu}}(\nu).
    \end{align*}
    Note that the summands vanish if $\min(n_{\mu}, n_{\nu}) = 1$. By comparing the coefficient of $a^{n_\mu - 2} b^{n_\nu - 2}$ for $p(t) = t^{n}$ in~(\ref{ccm_identity_alt2}), we see that it is necessary and sufficient that
    \begin{align}
        &\widetilde{m}_{\mu, \nu}((t_{\mu}, t_{\nu}) \mapsto t_{\mu}^{n_{\mu} - 2} t_{\nu}^{n_{\nu} - 2}) = \sum_{k_{\mu}, k_{\nu} \geq 1} \frac{(n_{\mu} - 2)! (n_{\nu} - 2)!}{k_{\mu}! k_{\nu}!} \label{series_precursor}\\
        & \left(\frac{1}{((n_{\mu} - k_{\mu}) + 1)! ((n_{\nu} - k_{\nu}) + 1)!} - \frac{1}{((n_{\mu} - k_{\mu}) + (n_{\nu} - k_{\nu}) + 1)!}\right) K_{n_{\mu}, k_{\mu}}(\mu) K_{n_{\nu}, k_{\nu}}(\nu).\nonumber
    \end{align}
    If one writes
    \begin{align*}
        J_{l}^\mu(n) = \sum_{k \geq 1} \frac{(n + 2 - k)!}{(n - k - l + 1)! (n - k + l + 4)!} \frac{n!}{k!}K_{n + 2, k}(\mu),
    \end{align*}
    then the combinatorics of Lemma~\ref{identity_lemma} imply that (\ref{series_precursor}) is equivalent to
    \begin{align*}
        \widetilde{m}_{\mu, \nu}((t_{\mu}, t_{\nu}) \mapsto t_{\mu}^{n_{\mu}} t_{\nu}^{n_{\nu}}) = \sum_{l = 0}^{\infty} (-1)^{l} (2 l + 3) J_{l}^\mu(n_{\mu}) J_{l}^\nu(n_{\nu}).
    \end{align*}
    It remains to verify that $J_{l}^\mu(n) = I_l^\mu\left(t \mapsto t^{n}\right)$, cf.~(\ref{basic_ccm_series}). Indeed,
    \begin{align*}
        J_{l}^\mu(n) &= \sum_{k \geq 1} \frac{(n + 2 - k)!}{(n - k - l + 1)! (n - k + l + 4)!} \frac{n!}{k!} \sum_{r = k}^{n} (-1)^{k - r} \frac{k}{r}\binom{n + r - k + 1}{r - k} M_{n + 2, r}(\mu)\\
        &= \frac{n!}{(n - l)! (n + l + 3)!}\sum_{r \geq 1} \frac{(-1)^{r - 1} (n + r)!}{r!} \pFq{3}{2}{l - n, -n - l - 3, 1 - r}{-n - 1, -n - r}{1} M_{n + 2, r}(\mu) \\
        &= \frac{n!}{(n - l)! (n + l + 3)!}\sum_{r \geq 1} \frac{(-1)^{r - 1}}{r!} \frac{(l + r)! (n - r + 2)!}{(l - r + 2)!} M_{n + 2, r}(\mu) = I_l^\mu\left(t \mapsto t^{n}\right),
    \end{align*}
    where we use Lemma \ref{cumulant_moment_identities} and Saalsch{\"u}tz's theorem~(\ref{saalschutz}) in the first and third equalities respectively.
\end{proof}

\subsection{Step 2}

\begin{prop}\label{prop2}
    Let $\mu$ be a probability measure supported on $d$ points, $A \in M_{d}(\C)$ be Hermitian, and $v \in \C^{d}$, such that $m_{k}(\mu) = \langle A^{k} v, v\rangle$ for any $k \in \mathbb{N}$. Then, for any $k, l \in \mathbb{N}$,
    \begin{align}\label{I_functional}
        I_l^\mu\left(t \mapsto t^{k}\right) = \int a^{k} C_{l}^{3/2}(2 b - 1) \omega_{A, v v^{*}}(a, b) \d{a} \d{b},
    \end{align}
    where $\omega_{A, v v^{*}}$ is as in Lemma~\ref{weird_moment_lemma}, and $C_l^{3/2}$ is a Gegenbauer polynomial as in section~\ref{hypergeometric}.
\end{prop}

\begin{proof}
    Since $\tr\left(A^{i_{1}} (v v^{*})^{j_{1}} \cdots A^{i_{n}} (v v^{*})^{j_{n}}\right) = m_{i_{1}}(\mu) m_{i_{2}}(\mu) \ldots, m_{i_{n}}(\mu)$, Lemma~\ref{weird_moment_lemma}, along with some combinatorics, implies that
    \begin{align*}
        \int a^{k} b^{l} \omega_{A, v v^{*}}(a, b) \d{a} \d{b} = \frac{1}{(k + l + 2) (k + l + 3)} \sum_{n = 1}^{k + 2} \frac{(-1)^{n - 1} \binom{l + 1}{n - 1}}{\binom{k + l + 1}{n - 1}} M_{k + 2, n}(\mu).
    \end{align*}
    By the series expansion for Gegenbauer polynomials~(\ref{gegenbauer_hypergeometric}), we get
    \begin{align*}
        & \int a^{k} C_{l}^{3/2}(2 b - 1) \omega_{A, v v^{*}}(a, b) \d{a} \d{b} \\
        &\quad=  \sum_{n = 1}^{k + 2}\sum_{l' = 0}^{l}\frac{1}{(k + l' + 2) (k + l' + 3)}\frac{(-1)^{n - 1} \binom{l' + 1}{n - 1}}{\binom{k + l' + 1}{n - 1}} \frac{(l + 3)^{(l')} (-l)^{(l')}}{2^{(l')} l'!} M_{k + 2, n}(\mu) \\
        &\quad= \sum_{n = 1}^{k + 2}\frac{1}{(l + 1)(l + 2)} \sum_{l' = 0}^{l} (-1)^{l' + n - 1}\frac{(k + l - n - l' + 2)!}{(k + l - l' + 3)! (l - n - l' + 2)!}\frac{(2 l - l' + 2)!}{l'! (l - l')!} M_{k + 2, n}(\mu) \\
        &\quad= \sum_{n = 1}^{k + 2}\frac{(-1)^{n- 1}}{(l + 1)(l + 2)} \frac{(k + l - n + 2)! (2 l + 2)!}{(k + l + 3)! (l - n + 2)! l!} \pFq{3}{2}{-k - l - 3, -l + n - 2, -l}{-k -l + n - 2, -2 l - 2}{1} M_{k + 2, n}(\mu)\\
        &\quad= \sum_{n = 1}^{k + 2}\frac{k!}{(k + l + 3)! (k - l)!}\frac{(-1)^{n - 1} (k - n + 2)!}{n!} \frac{(l + n)!}{(l - n + 2)!}  M_{k + 2, n}(\mu) = I_l^\mu\left(t \mapsto t^{k}\right),
    \end{align*}
    where we used Saalsch{\"u}tz's theorem~(\ref{saalschutz}) in the penultimate equality.
\end{proof}

For a finitely-supported probability measure $\mu$, we write $\omega_{\mu} = \omega_{A, v v^{*}}$, with $A$ and $v$ as in Proposition~\ref{prop2}.

\begin{prop}\label{pointwise_bound}
    Let $A \in M_{d}(\C)$ be Hermitian, $v \in \C^{d}$ with $\|v\|_{2} = 1$, and $a, b \in \R$. Then, the matrix
    \begin{align}\label{lol3}
        (A - a I) (v v^{*} - b I)
    \end{align}
    has at most one conjugate pair of non-real eigenvalues, and their imaginary parts are bounded by $\length(\conv(\lambda(A)))$. Consequently,
    \begin{align}\label{lol5}
        \omega_{\mu}(a, b) \leq \frac{1}{\pi}\length(\conv(\supp(\mu))),
    \end{align}
    for any finitely-supported probability measure $\mu$.
\end{prop}

\begin{proof}
    The rank-$1$ update formula for determinants \cite[0.8.5.11]{MR2978290} implies that the eigenvalues $z$ of (\ref{lol3}) satisfy
    \begin{align}\label{lol4}
        0 = \frac{1 - b}{(a - \frac{z}{b}) - a} - \left\langle \left(\left(a - \frac{z}{b}\right) I - A\right)^{-1} v, v\right\rangle = \frac{1 - b}{w - a} -  \left\langle (w I - A)^{-1} v, v \right\rangle,
    \end{align}
    where $w = a - z/b$.
    Observe that (\ref{lol4}) is a rational function of degree at most $d + 1$ in $w$, and has at most $d$ finite roots. On the interval between two consecutive eigenvalues of $A$, the rational function is unbounded from above and below; so, it has a real root on the interval unless $a$ is on the interval. Hence, at least $d - 2$ of the roots are real. This implies the claim on the upper bound for the non-real roots. We also conclude that there are no non-real roots if $a \notin \conv(\lambda(A))$.

    For the bound on the imaginary parts, (\ref{lol4}) can be rewritten as
    \begin{align*}
        -b = \left\langle (w I - A)^{-1} (A - a I) v, v \right\rangle.
    \end{align*}
	If $|\Im(w)| > \length(\conv(\lambda(A)))/b$,
    the operator inside the quadratic form has operator norm less than $b$, and hence $w$ is not a solution to (\ref{lol4}). But $|\Im(z)| = b |\Im(w)|$, which gives the claim.
\end{proof}

\subsection{Step 3}

\begin{prop}\label{prop3}
    Let $\mu$ and $\nu$ be finitely-supported probability measures.

    Consider the function
    \begin{align}
        w_{\mu, \nu}(t_{\mu}, t_{\nu}) = \int_{0}^{1} \int_{0}^{1} \min\left(\frac{1}{(1 - b_{\mu}) (1 - b_{\nu})} , \frac{1}{b_{\mu} b_{\nu}}\right) \omega_{\mu}(t_{\mu}, b_{\mu}) \omega_{\nu}(t_{\nu}, b_{\nu}) \d{b_{\mu}} \d{b_{\nu}}.
    \end{align}
    Then $w_{\mu, \nu}$ is non-negative, continuous, and supported on $\conv(\supp(\mu)) \times \conv(\supp(\nu))$, and for any $n_{\nu}, n_{\mu} \in \mathbb{N}$, we have
    \begin{align}\label{ccm_formula}
        \widetilde{m}_{\mu, \nu}((t_{\mu}, t_{\nu}) \mapsto t_{\mu}^{n_{\mu}} t_{\nu}^{n_{\nu}}) = \int \int t_{\mu}^{n_{\mu}} t_{\nu}^{n_{\nu}} w_{\mu, \nu}(t_{\mu}, t_{\nu})\d{t_{\mu}} \d{t_{\nu}}.
    \end{align}
\end{prop}

\begin{proof}
    Let
    \begin{align*}
        g(b_{\mu}, b_{\nu}) = \min\left(\frac{1}{(1 - b_{\mu}) (1 - b_{\nu})} , \frac{1}{b_{\mu} b_{\nu}}\right).
    \end{align*}
    
    All but the last property of $w_{\mu, \nu}$ in the statement follow from analogous properties of  $\omega_{\mu}$ and $\omega_{\nu}$, and the observation that
    \begin{align}\label{finite_integral1}
        \int_{0}^{1} \int_{0}^{1}g(b_{\mu}, b_{\nu}) \d{b_{\mu}} \d{b_{\nu}} \leq \left( \int_{0}^{1} \frac{1}{\sqrt{b (1 - b)}} \d{b} \right)^{2}< \infty.
    \end{align}

    Our idea is to show that in a suitable $L_2$-space,
    \begin{align}\label{g_series}
        g(b_{\mu}, b_{\nu}) = \sum_{k_{\mu}, k_{\nu}} c_{k_{\mu}, k_{\nu}} C_{k_{\mu}}^{3/2}(2 b_{\mu} - 1) C_{k_{\nu}}^{3/2}(2 b_{\nu} - 1),
    \end{align}
    and then compute the coefficients $c_{k_{\mu}, k_{\nu}}$.

    We consider the Hilbert space $L_{2}([0, 1]^{2}, h)$, where $h(b_{\mu}, b_{\nu}) := b_{\mu} (1 - b_{\mu}) b_{\nu}(1 - b_{\nu})$. Equation~(\ref{gegenbauer_orthogonality}) says that the products of the Gegenbauer polynomials $C_{k_{\mu}}^{3/2}(2 b_{\mu} - 1) C_{k_{\nu}}^{3/2}(2 b_{\nu} - 1)$ form an orthogonal basis of $L_{2}(h)$. It remains to check that $g \in L_{2}(h)$; this follows from a pointwise comparison to (\ref{finite_integral1}).
    
    We can compute the right-hand side of (\ref{ccm_formula}),
    \begin{align}\label{lol6}
        & \int \int t_{\mu}^{n_{\mu}} t_{\nu}^{n_{\nu}} w_{\mu, \nu}(t_{\mu}, t_{\nu})\d{t_{\mu}} \d{t_{\nu}} \\
        &\qquad= \int_{0}^{1} \int_{0}^{1} g(b_\mu, b_\nu) \left(\int t_{\mu}^{n_{\mu}} \omega_{\mu}(t_{\mu}, b_{\mu}) \d{t_{\mu}} \right) \left( \int t_{\nu}^{n_{\nu}}\omega_{\nu}(t_{\nu}, b_{\nu}) \d{t_{\nu}}\right) \d{b_{\mu}} \d{b_{\nu}}. \nonumber
    \end{align}
    Orthogonality of the Gegenbauer polynomials and Proposition~\ref{prop2} hence imply that (\ref{lol6}) equals the finite sum
    \begin{align*}
        &\sum_{k_{\mu}, k_{\nu}} c_{k_{\mu}, k_{\nu}}\left(\int_{0}^{1}\int t_{\mu}^{n_{\mu}} C_{k_{\mu}}^{3/2}(2 b_{\mu} - 1)\omega_{\mu}(t_{\mu}, b_{\mu}) \d{t_{\mu}} \d{b_{\mu}}\right) \left(\int_{0}^{1}\int t_{\nu}^{n_{\nu}} C_{k_{\nu}}^{3/2}(2 b_{\nu} - 1)\omega_{\nu}(t_{\nu}, b_{\nu}) \d{t_{\nu}} \d{b_{\nu}}\right) \\
        &\qquad\qquad= \sum_{k_{\mu}, k_{\nu}} c_{k_{\mu}, k_{\nu}} I_{k_{\mu}}^{\mu}(t_{\mu} \mapsto t_{\mu}^{n_{\mu}}) I_{k_{\nu}}^{\nu}(t_{\nu} \mapsto t_{\nu}^{n_{\nu}}).
    \end{align*}
    In light of the definition of $\widetilde{m}_{\mu,\nu}$ (\ref{basic_ccm_series}), it remains to check that
    \begin{align}\label{g_coeffs}
        c_{k_{\mu}, k_{\nu}} = \delta_{k_{\mu}, k_{\nu}} (2 k_{\mu} + 3) (-1)^{k_{\mu}}.
    \end{align}
    
    For any $|t_{\mu}| < 1$, using the generating function of the Gegenbauer polynomials (\ref{gegenbauer_series}) and orthogonality (\ref{gegenbauer_orthogonality}), we have
    \begin{align*}
         &\sum_{k_{\mu}, k_{\nu}} c_{k_{\mu}, k_{\nu}}\frac{4 t_{\nu}^{k_{\nu}}}{2 k_{\nu} + 3} C^{3/2}_{k_{\mu}}(x_{\mu}) = \int_{-1}^{1} \sum_{k_{\mu}, k_{\nu}} c_{k_{\mu}, k_{\nu}} C_{k_{\mu}}^{3/2}(x_{\mu}) C_{k_{\nu}}^{3/2}(x_{\nu}) (1 - 2 x_{\mu} t_{\mu} + t_{\mu}^{2})^{-3/2} (1 - x_{\nu}^{2})\d{x_{\nu}} \\
        &\qquad\qquad= \int_{-1}^{1} 4 \min\left(\frac{1}{(1 + x_{\mu}) (1 + x_{\nu})}, \frac{1}{(1 - x_{\mu}) (1 - x_{\nu})}\right)(1 - 2 x_{\mu} t_{\mu} + t_{\mu}^{2})^{-3/2} (1 - x_{\nu}^{2})\d{x_{\nu}}.
    \end{align*}
    The last integral can be evaluated using the  antiderivative
    \begin{align*}
        \int \frac{a + b x}{(1 - \alpha x)^{3/2}} = \frac{2 (b + a \alpha + b (1 - \alpha x))}{\alpha^{2} (1 - \alpha x)^{1/2}}.
    \end{align*}
    With this, we obtain
    \begin{align*}
        \sum_{k_{\mu}, k_{\nu}} c_{k_{\mu}, k_{\nu}}\frac{4 t_{\nu}^{k_{\nu}}}{(2 k_{\nu} + 3)} C_{k_{\mu}}^{3/2}(x_{\mu}) = \frac{8 ((1 + 2 t_{\nu} x_{\mu} + t_{\nu}^{2})^{1/2} - (1 + 2 t_{\nu} x_{\mu}))}{t_{\nu}^2 (1 - x_{\mu}^2)}.
    \end{align*}
    Multiplying both sides by $t_{\nu}^{2}$, and differentiating twice in $t_{\nu}$, yields
    \begin{align*}
        \sum_{k_{\mu}, k_{\nu}} c_{k_{\mu}, k_{\nu}}\frac{4 t_{\nu}^{k_{\nu}} (k_{\nu} + 1) (k_{\nu} + 2)}{(2 k_{\nu} + 3)} C_{k_{\mu}}^{3/2}(x_{\mu}) &= \frac{8}{(1 + 2 t_{\nu} x_{\mu} + t_{\nu}^{2})^{3/2}} \\
        &= 4 \sum_{k_{\mu} = 0}^{\infty} (k_{\mu} + 1) (k_{\mu} + 2) C_{k_{\mu}}^{3/2}(x_{\mu}) (-1)^{k_{\mu}} t_{\nu}^{k_{\mu}}.
    \end{align*}
    It remains to compare the coefficients of $C_{k_{\mu}}^{3/2}(x_{\mu}) t_{\nu}^{k_{\nu}}$.
\end{proof}

\subsection{Step 4}

\begin{proof}[Proof of Theorem \ref{thm2}]
    Propositions~\ref{prop1} and \ref{prop3} together verify that $\widetilde{m}_{\mu, \nu}$ is a positive measure for finitely-supported $\mu$ and $\nu$. Let us approximate general $\mu$ and $\nu$ by finitely-supported measures in distribution, to obtain a sequence $\widetilde{m}_{\mu_{n}, \nu_{n}}$ of probability measures with uniformly bounded total measures and supports in $\conv(\supp(\mu)) \times \conv(\supp(\nu))$. The bound is indeed $1/12 \Var(\mu) \Var(\nu)$. Weak-$*$ compactness of measures implies the existence of a limiting measure $\widetilde{m}_{\mu, \nu}$ with the desired properties, at least for polynomial $f$. Approximating again deals with general $f \in C^{4}(\R)$.
\end{proof}

\section{Comments and conjectures}\label{discussion}

\subsection{Finite free convolution and c-convolution}
    The author's original motivation was to understand a variant of inequality (\ref{moment_comparison}) that compares free convolution to \emph{finite free convolution} of polynomials. Finite free convolution gained traction with the work of Marcus, Spielman, and Srivastava \cite{MR4408504}. The finite free additive convolution $p \boxplus_{d} q$ of two degree $d$ polynomials with real roots is also a degree $d$ polynomial with real roots. If $\mu_{p}$ and $\mu_{q}$ are probability measures with equal point masses at the roots of $p$ and $q$ respectively, the distribution $\mu_{p \boxplus_{d} q}$, for large $d$, is close to $\mu_{p} \boxplus \mu_{q}$ in a suitable sense \cite[Corollary 5.3]{MR3741428}. It follows from the work of Marcus, Spielman, and Srivastava \cite[Theorem 4.2]{srivastava2024ramanujan} that the supports of finite free convolution and free convolution can in fact be compared,
    \begin{align*}
        \conv \{x \mid (p \boxplus_{d} q)(x) = 0\} \subset \conv(\supp(\mu_{p} \boxplus \mu_{q})).
    \end{align*}
    Spielman allegedly asked whether this bound can be generalized to all even moments; namely, whether for $k \in \mathbb{N}$,
    \begin{align}\label{nikhil_question}
        \frac{1}{d}\sum_{x : (p \boxplus_{d} q)(x) = 0} x^{2 k} \leq m_{2k}(\mu_{p} \boxplus \mu_{q}).
    \end{align}

    While we were not able to answer this question, it lead to Theorems~\ref{thm1} and \ref{thm2}. It is conceivable that Theorem~\ref{thm1} (and \ref{thm2}) has a variant in the finite free setting.

    \begin{conj}\label{finite_free_4}
        Let $p$ and $q$ be real-rooted polynomials, and $f \in C^{4}(\R)$ with $f^{(4)} \geq 0$. Then,
        \begin{align*}
            \frac{1}{d}\sum_{x : (p \boxplus_{d} q)(x) = 0} f(x) \leq \int f(t) \d{(\mu_{p} \boxplus \mu_{q})}(t).
        \end{align*}
    \end{conj}

    It is possible to modify the arguments of section~\ref{necessity} to show that Conjecture~\ref{finite_free_4} is true up to leading order, and the condition $f^{(4)} \geq 0$ is necessary.

    Another avenue for generalization of our results is towards the so-called $c$-convolution of Mergny and Potters \cite{MR4367356}, $\mu \oplus_{c} \nu$, which convolves two compactly-supported probability measures $\mu$ and $\nu$, and is parametrized by $c \in [0, \infty]$. Classical and free convolutions of $\mu$ and $\nu$ can be recovered as $\mu \oplus_{0} \nu$ and $\mu \oplus_{\infty} \nu$, respectively. While $c$-convolution can be formally defined, either analytically or with suitable cumulants, it is not known whether $\mu \oplus_{c} \nu$ is always a probability measure. We can still make the following conjecture.
    \begin{conj}\label{c_free_4}
        Let $\mu$ and $\nu$ be compactly-supported probability measures, and $f \in C^{4}(\R)$ with $f^{(4)} \geq 0$. Then,
        \begin{align*}
            c \mapsto \int f(t) \d{(\mu \oplus_{c} \nu)}(t)
        \end{align*}
        is decreasing in $c$.
    \end{conj}
    Again, one can show that this conjecture is true to the leading order, and the condition $f^{(4)} \geq 0$ is necessary.

    In fact, $c$-convolution is closely related to finite free convolution, and formally $\mu_{p \boxplus_{d} q} = \mu_{p} \oplus_{-d} \mu_{q}$ \cite{MR4367356}.

\subsection{The 4-ordering of probability measures}

Theorem~\ref{thm1} suggests the following ordering on probability measures.

\begin{defn}
    Let $\mu$ and $\nu$ be probability measures such that
    \begin{align}\label{4_ordering_moment_bound}
        \int |t|^{3} \d{\mu}(t) < \infty, \int |t|^{3} \d{\nu}(t) < \infty.
    \end{align}
    We write $\mu \prec_{4} \nu$
    if, for any $f \in C^{4}(\R)$ with $f^{(4)} \geq 0$ and $f^{(4)}$ compactly supported,
    \begin{align}\label{4_ordering}
        \int f(t) \d{\mu}(t) \leq \int f(t) \d{\nu}(t).
    \end{align}
\end{defn}

Condition (\ref{4_ordering_moment_bound}) ensures that both sides of (\ref{4_ordering}) are integrable.

Theorem~\ref{thm1} is equivalent to the claim that $\mu \boxplus \nu \prec_{4} \mu * \nu$, for compactly supported $\mu$ and $\nu$. The $4$-ordering further satisfies the following.

\begin{prop}\label{4_ordering_classic}
    Let $\mu_{1}, \mu_{2}$, and $\nu$ be compactly-supported probability measures, with $\mu_{1} \prec_{4} \mu_{2}$, then
    \begin{align*}
        \mu_{1} * \nu \prec_{4} \mu_{2} * \nu.
    \end{align*}
\end{prop}
\begin{proof}
    \begin{align*}
        \int f(t) \d{(\mu_{2} * \nu)}(t) - \int f(t) \d{(\mu_{1} * \nu)}(t) = \int \left(\int f(t - s) \d{\mu_{2}}(s) - \int f(t - s) \d{\mu_{1}}(s) \right) \d{\nu}(t).
    \end{align*}
\end{proof}

\begin{prop}\label{thm1_multiple}
    Let $\mu_{1}, \mu_{2}, \ldots, \mu_{n}$ be compactly-supported probability measures. Then,
    \begin{align*}
        \mu_{1} \boxplus \mu_{2} \boxplus \cdots \boxplus \mu_{n} \prec_{4} \mu_{1} * \mu_{2} * \cdots * \mu_{n}.
    \end{align*}
\end{prop}

\begin{proof}
    Using Theorem~\ref{thm1} and induction,
    \begin{align*}
        (\mu_{1} \boxplus \mu_{2} \boxplus \cdots \boxplus \mu_{n - 1}) \boxplus \mu_{n} \prec_{4} (\mu_{1} \boxplus \mu_{2} \boxplus \cdots \boxplus \mu_{n - 1}) * \mu_{n} \prec_{4} (\mu_{1} * \mu_{2} * \cdots *\mu_{n - 1}) * \mu_{n}.
    \end{align*}
\end{proof}

If we apply Proposition~\ref{thm1_multiple} to $\mu_{1} = \mu_{2} = \ldots = \mu_{n} = \mu$ for a compactly-supported probability measure $\mu$, and consider its large-$n$ asymptotics, the central limit theorem and free central limit theorem \cite[2.1]{MR3585560} imply that, in the limit,
\begin{align*}
    \mu_{sc} \prec_{4} \gamma,
\end{align*}
where $\mu_{sc}$ is the standard semicircular measure, and $\gamma$ is the standard Gaussian.

\begin{kys}
    If $\mu_{1}, \mu_{2}$, and $\nu$ are compactly-supported probability measures such that $\mu_{1} \prec_{4} \mu_{2}$, is $\mu_{1} \boxplus \nu \prec_{4} \mu_{2} \boxplus \nu$?
\end{kys}

What other properties does $\prec_{4}$ have? 

\subsection{Miscellanea}

\begin{itemize}
    \item What happens in Theorems~\ref{thm1} and \ref{thm2} for unbounded $\mu$ and $\nu$? What if $\mu$ and $\nu$ don't have any finite moments?

    \item Does Theorem \ref{thm1} have a variant for free multiplicate convolution, a notion introduced by Voiculescu \cite{MR915507}? More generally, does Corollary \ref{cor1} have variants for other, non-symmetric mixed moments?

    \item Can one interpret the density $\omega_{A, B}$ of Lemma~\ref{weird_moment_lemma} in finite tracial von Neumann algebras? That is, does a positive measure exist with moments given by (\ref{moment_expression})?

    \item Can Proposition~\ref{pointwise_bound} be improved? How about a pointwise bound for $\omega_{A, B}$ when $B$ isn't necessarily rank one?
\end{itemize}

\section{Acknowledgements}

I thank Roland Speicher, Nikhil Srivastava, Joel Tropp, and Jorge Garza-Vargas for helpful discussions and comments, and Jorge Garza-Vargas in particular for introducing me to the conjectural inequality~(\ref{nikhil_question}). Also, shikhin for editing.

\bibliography{refs_paper}

\bibliographystyle{alpha}

\end{document}